\documentclass[a4paper,10pt]{amsart}

\usepackage[english]{babel}
\usepackage[utf8]{inputenx}
\usepackage[T1]{fontenc} 
\usepackage{graphicx}         
\usepackage{xcolor}  
\usepackage{hyperref}
\usepackage{tikz-cd}
\usepackage{lmodern}
\usepackage{verbatim}
\usepackage{longtable}

\numberwithin{equation}{section}

\title[Upper bounds on class numbers]{Upper bounds on class numbers of real quadratic fields}

\author{Riccardo Bernardini}
\address{Dipartimento di Matematica Guido Castelnuovo, Sapienza Università di Roma, Piazzale Aldo Moro 5, 00185 Rome, Italy}
\email{r.bernardini@uniroma1.it}
\date{\today} 

\subjclass[2020]{Primary 11R29; Secondary 11R11, 11R27, 11A55.}  
\keywords{Real quadratic fields, class number, fundamental unit, L-functions, quadratic orders, continued fractions}

\begin{document}

\newtheorem{thm}{Theorem}[section]
\newtheorem{defn}[thm]{Definition}
\newtheorem{lem}[thm]{Lemma}
\newtheorem{cor}[thm]{Corollary}
\newtheorem{prop}[thm]{Proposition}
\newtheorem{ex}[thm]{Example}
\newtheorem{rmk}[thm]{Remark}
\theoremstyle{definition}

\begin{abstract}
    We prove that, for any $\varepsilon>0$, the number of real quadratic fields $\mathbb{Q}(\sqrt{d})$ of discriminant $d<x$ whose class number is $\ll \sqrt{d}(\log{d})^{-2}(\log\log{d})^{-1}$ is at least $x^{1/2-\varepsilon}$ for $x$ large enough. This improves by a factor $\log\log{d}$ a result from 1971 by Yamamoto. We also establish a similar estimate for $m$-tuples of discriminants for any $m\geq 1$. 
Finally, we provide algebraic conditions to give a lower bound for the size of the fundamental unit of $\mathbb{Q}(\sqrt{d})$, generalizing a criterion by Yamamoto. Our proof corrects a work of Halter-Koch.
\end{abstract}

\maketitle

\section{Introduction}\label{Intro}
The class number is one of the most studied invariants of a number field and its behavior in relation to the discriminant is not completely understood, even for quadratic fields. If the quadratic field is real the class number formula reduces to
\begin{equation}\label{class number formula}
h(d)=\frac{\sqrt{d}L(1,\chi_d)}{2\log{\xi_d}},
\end{equation}
where $d>0$ is the discriminant, $\chi_d$ is Kronecker's character and $\xi_d$ the fundamental unit. 
The problem of controlling, at least asymptotically, the size of the class number of quadratic fields dates back to Littlewood \cite{Littlewood} who bounded, under GRH, the $L$-function for any fundamental discriminant $d$ 
\begin{equation}\label{bound L}
\frac{1}{\log{\log{|d|}}}\ll L(1,\chi_d)\ll\log{\log{|d|}};
\end{equation}
more precisely, he gave the explicit estimate
\begin{equation}\label{stima L Littlewood}
L(1,\chi_d) \leq (2e^\gamma +o(1))\log\log{|d|}.
\end{equation}
He then obtained
\begin{equation*}
h(d)\leq \left( \frac{2e^{\gamma}}{\pi}+o(1) \right)\sqrt{|d|}\log{\log{|d|}},
\end{equation*}
where $d<0$ and $\gamma$ is the Euler--Mascheroni constant.
For real quadratic fields ($d>0$) the trivial inequality $\xi_d\geq \sqrt{d}/2$ and \eqref{stima L Littlewood} imply
\begin{equation}\label{real Littlewood}
h(d)\leq \left( 4e^{\gamma}+o(1) \right)\sqrt{d}\frac{\log{\log{d}}}{\log{d}},
\end{equation}
valid under the GRH.
It is however believed that \eqref{real Littlewood} holds with a smaller constant, that is
\begin{equation}\label{eq Lamzouri}
h(d)\leq \left(2e^{\gamma}+o(1) \right)\sqrt{d}\frac{\log{\log{d}}}{\log{d}}.
\end{equation}

Recently Lamzouri \cite[Theorem 1.2]{Lamzouri} calculated the number of positive discriminants $d\leq x$ for which the class number is bigger than the right-hand side of \eqref{eq Lamzouri} finding that, if $x$ is large enough, there are at least $x^{1/2-1/\log{\log{x}}}$ of them.
He developed a previous work by Montgomery and Weinberger \cite{Mont-Wein} where they studied discriminants arising from Chowla's family, i.e. $d=4n^2+1$. These have the peculiarity that the fundamental unit is essentially as small as possible, namely $\xi_d\leq2\sqrt{d}$. Therefore \eqref{class number formula} implies 
\begin{equation*}
h(d)=(1+o(1))\frac{\sqrt{d}}{\log{d}}L(1,\chi_d).
\end{equation*}
Montgomery and Weinberger showed the existence of infinitely many squarefree $d$ for which $L(1,\chi_d)\gg\log\log{d}$. Their proof relies on the fact that large values of the $L$-function are reached if $\chi_d(p)=1$ for all small primes.

The problem of counting the number of discriminants $d\leq x$ satisfying certain inequalities has also been investigated by other authors. For instance  Lamzouri and Dahl \cite{Dahl-Lam} showed that the number of Chowla's discriminants $d\leq x$ for which 
\begin{equation*}
h(d)\geq 2e^\gamma \frac{\sqrt{d}}{\log{d}} \tau\quad\mbox{and}\quad h(d)\leq 2e^{-\gamma}\zeta(2)\frac{\sqrt{d}}{\log{d}}\frac{1}{\tau}
\end{equation*}
hold, where $\tau$ can be chosen $\leq \log\log{x}-3\log\log\log{x}$, has a positive proportion, which was made explicit by the authors.
Later, this result has been generalized by Dahl and Kala \cite[Theorem 1]{Dahl-Kala} to families of discriminants $d\equiv 1\pmod{4}$ such that
\begin{equation*}
\frac{1+\sqrt{d}}{2}=[f(n),\overline{u_1,\ldots, u_{s}, 2f(n)-1}]
\end{equation*}
for some linear polynomial $f(n)$, where the symmetric sequence $u_1,\ldots, u_{s}$ is fixed. 

In this paper we prove the existence of infinitely many discriminants $d<x$ whose class number is $\ll \sqrt{d}(\log{d})^{-2}(\log\log{d})^{-1}$.


\begin{thm}\label{thm Yam family}
Let $\varepsilon>0$. Then, for $x$ big enough, there are at least $x^{1/2-\varepsilon}$ fundamental discriminants $0<d<x$ such that
\begin{equation}\label{asymptotic growth}
h(d)\ll \frac{\sqrt{d}}{(\log{d})^2 \log\log{d}}.
\end{equation}
\end{thm}

In the proof of Theorem \ref{thm Yam family}, we consider discriminants of the form $d=n^2+4p$ where $p$ is a prime. As shown by Yamamoto  \cite{Yamamoto}, they have a fundamental unit satisfying $\log{\xi_d}\gg(\log{d})^2$. Moreover, in this family, we are able to determine the squarefree values, since $d$ is expressed by a quadratic polynomial, and to establish the bound
\begin{equation*}
L(1,\chi_d)\ll \frac{1}{\log\log{d}}.
\end{equation*} 
We explicitly compute the implied constant in \eqref{asymptotic growth} and we show one can choose $192\log{p}$. Similarly, for any $\varepsilon>0$ there are at least $x^{1/2-\varepsilon}$ fundamental discriminants $0<d<x$ in Chowla's family such that
\begin{equation}\label{prop d Chowla}
h(d)\ll  \frac{\sqrt{d}}{\log{d} \log\log{d}};
\end{equation}
Theorem \ref{thm Yam family} improves \eqref{prop d Chowla} by $(\log{d})^{-1}$.

We remark that \eqref{asymptotic growth} still holds if one takes $d=n^2-4p$, for which again $\log{\xi_d} \gg(\log{d})^2$, see \cite{Yamamoto}. The case $p=2$ is of particular importance since $n^2-8$, known as Shanks' family, contains a subfamily with explicit fundamental unit. If we take $d=(2^k+3)^2-8$ squarefree, in \cite{Yamamoto} is shown that
\begin{equation*}
\xi_d=\left(\frac{2^k+3+\sqrt{d}}{4}\right)^k\frac{2^k+1+\sqrt{d}}{2};
\end{equation*}
thus $\log{\xi_d}\asymp (\log{d})^2$.

Theorem \ref{thm Yam family} can be generalized to simultaneous discriminants $d_i=N+4p_i$ for $i=1,\ldots,m$ with $p_i$ distinct primes and $N=n^2$ a perfect square.


\begin{thm}\label{thm simultaneo}
Fix $\varepsilon>0$, let $m$ be a positive integer and $a\in\mathbb{R}^+$. 
Moreover, let $p_i$ for $i=1,\ldots,m$ be distinct primes such that, if $d_i:=N+4p_i$ with $N$ an integer parameter,
\begin{itemize}
   \item[$(*)$] there exists a quadratic residue $N_0$ modulo $\prod_{p_j\leq 2m} p_j$ for which $(d_i,p_j)=1$ for all $i$ and for all $p_j\leq 2m$.
\end{itemize}
Then, if $x\gg_{\varepsilon,m,a} 1$ and $p_i\leq (\log{x})^a$ for all $i$, there are at least $x^{1/2-\varepsilon}$ positive integers $N<x$ such that each $d_i$ is a fundamental discriminant and
\begin{equation*}
h(d_i)\leq C(m) \log{\left(\max_i\{p_i\}\right)} \frac{\sqrt{N}}{(\log{N})^2\log\log{N}},
\end{equation*}
where 
\begin{equation*}
C(m):=16\prod_{p_i> m^24^m+2}\left( \frac{p_i+1}{p_i-1} \right)\prod_{p\leq m^24^m+2}\left( \frac{p+1}{p-1} \right).
\end{equation*}
\end{thm}

For $m=1$ $(*)$ is empty if $p_1>2$ and reduces to $n^2+8\equiv 1 \pmod{2}$ if $p_1=2$; then Theorem \ref{thm simultaneo} becomes a generalization of Theorem \ref{thm Yam family}.
The hypothesis $(*)$ is involved in a technical step to use the small sieve.

We point out that the primes $p_i$ cannot be chosen arbitrarily as we need them to be unramified in order to apply \cite[Theorem 3.1]{Yamamoto}. For example, if $p_1=3$, $p_2=11$, $p_3=17$ and $p_4=23$, one then sees that $d_i$ is always a multiple of $9$ for some $i$ and that $(*)$ is not satisfied by these $p_i$'s. One can also show that for any positive integer $m$ it is possible to find $p_1,\ldots,p_m$ for which $(*)$ holds. Indeed, by induction and Dirichlet's theorem on arithmetic progressions we can exhibit $m$ primes $p_i\equiv1\pmod{4}$ such that the Legendre symbol $\left(\frac{-p_j}{p_i}\right)$ is $-1$ for every $i< j$. By quadratic reciprocity one sees
\begin{equation*}
\left(\frac{-p_j}{p_i}\right)=\left(\frac{-p_i}{p_j}\right)
\end{equation*}
and then $\left(\frac{-p_j}{p_i}\right)=-1$ for every $i\neq j$.
In order to fulfill $(*)$ it is now sufficient to take for $N$ the square of any invertible class modulo $\prod_{p_i\leq 2m} p_i$. One can also take $p_1=2$ simply by following the steps above for $p_2,\ldots,p_m$ and choosing $N$ odd. 

A natural way to try to improve Theorem \ref{thm Yam family} is to work with discriminants $d$ with bigger fundamental unit. Yamamoto \cite[Theorem 3.2]{Yamamoto} proved that, if 
\begin{equation}\label{d Yam log3}
m_k=(p^kq+p+1)^2-4p
\end{equation}
with $p<q$ prime numbers, the quadratic fields $\mathbb{Q}(\sqrt{m_k})$ satisfy $\log{\xi_{d_k}}\gg (\log{d_k})^3$, where $d_k$ the discriminant of $\mathbb{Q}(\sqrt{m_k})$. One can control $L(1,\chi_{d_k})$ by showing that it is $\ll\log{d_k}$, obtaining
\begin{equation}\label{Yam h estimate}
h(d_k)\ll \frac{\sqrt{d_k}}{(\log{d_k})^2}.
\end{equation}
Theorem \ref{thm Yam family} improves upon \eqref{Yam h estimate} by a factor $(\log\log{d})^{-1}$; it seems difficult to refine the bound of $L(1,\chi_{d_k})$ so that it produces an estimate stronger than that of Theorem \ref{thm Yam family}.

About twenty years after the paper of Yamamoto, Halter-Koch \cite{H-K} proposed a family of discriminants $d$ satisfying
\begin{equation}\label{ineq log 4}
\log{\xi_d}\gg (\log{\sqrt{d}})^4;
\end{equation}
as pointed out in \cite[p. 634]{JacWill}, an error in Halter-Koch's proof invalidates his result. In Section \ref{proof HKY criterion} we fix Halter-Koch's criterion, but we are still unable to find a family satisfying inequality \eqref{ineq log 4}. 

In what follows, $d$ is a quadratic discriminant, i.e. $d$ is a non-square integer $\equiv 0,1\pmod{4}$ not necessarily squarefree and $\mathcal{O}_d$ is the quadratic order associated to $d$ in the field $\mathbb{Q}(\sqrt{d})$. When $d$ is a fundamental discriminant, $\mathcal{O}_d$ coincides with the integral closure of $\mathbb{Z}$ in the quadratic field. For more details we refer the reader to Section \ref{prenot}. 

\begin{thm}\label{HKY criterion}
Let $n_i$ be in $\mathbb{N}_{\geq2}$ for $i=1,\ldots, m$ and let $\mathfrak{D}$ be an infinite set of quadratic discriminants $d>0$ satisfying the following conditions:
\begin{itemize}
    \item[(a)] $ n_1, \ldots, n_m $ are norms of reduced principal ideals of $ \mathcal{O}_d $;
    \item[(b)] each $n_i$ decomposes in the form $n_i = n_i'n_i''$,
    where $n_i'$ and $n_i''$ are natural numbers such that:
    \begin{itemize}
        \item[(b1)] $n_i'$ is coprime with $d$;
        \item[(b2)] $n_i''$ is a square-free divisor of the fundamental discriminant $ d_0$ associated with $d$.
        \item[(b3)] $n_1', \dots, n_m' $ are coprime to each other.
    \end{itemize}
\end{itemize}
Then it follows that
\begin{equation*}
\log{\xi_d} \gg (\log{\sqrt{d}})^{m+1} \quad (d \in \mathfrak{D}).
\end{equation*}
\end{thm}

Halter-Koch asked that $n_i'$ contained only primes $p$ such that $\chi_d(p) = 1$. As a consequence of \eqref{cond modulo ideale} and the fact that, by (a), the numbers $n_i$ in particular are norms of primitive ideals (see Section \ref{prenot}), $\chi_d(p) = 1$ is equivalent to (b1). Concerning (b3), in his version $n_1', \dots, n_m' $ were multiplicatively independent, i.e. if $ \prod_{i=1}^m (n_i')^{\alpha_i} = 1 $ with $ \alpha_i \in \mathbb{Z} $ then $\alpha_i=0$ for all $i$, while we opted for a more restrictive assumption.

When $n_i$ are pairwise distinct primes (hence $n_i''=1$) that split into principal ideals for every $d\in\mathfrak{D}$, we recover Yamamoto's \cite[Theorem 3.1]{Yamamoto}. 

Jacobson and Williams \cite[Theorem 3.1]{JacWill} also provided a refinement of the work of Yamamoto: given infinitely many real quadratic fields $K$, in order to have a large regulator, they required the existence of $m$ positive integers $r_i$ for $i=1,\ldots,m$, coprime in pairs, such that the ideals $(r_i)$ factor in $K$ into two principal ideals $R_i$ and $R'_i$. Moreover, each prime $p$ dividing $r_i$ splits. 
However, their statement appears to require some additional hypotheses on $R_i$ and $R'_i$. In fact, if $m=1$ and $R_1=(n_1)$ with $n_1\in\mathbb{N}$ (or equivalently $R'_1=(n'_1)$ with $n'_1\in\mathbb{N}$), we could find a counterexample taking a suitable subfamily of Chowla's discriminants. On the other hand, if $R_i$ (or equivalently $R'_i$) is not generated by an integer, the theorem is true and, more specifically, it is equivalent to Theorem \ref{HKY criterion} in the case of a succession of fundamental discriminants $d$. 
Therefore, Theorem \ref{HKY criterion} can be viewed both as a clarification and a generalization to quadratic orders of \cite[Theorem 3.1]{JacWill}.

Instead of asking for properties on the norms of principal and reduced ideals as we did, one can get the same bound on the regulator by assuming that the ideals are principal and independent; see \cite[p. 60]{Patt} for the definition of independent ideals and \cite[Theorem 27]{Patt}.

The interest in results like Yamamoto's criterion and Theorem \ref{HKY criterion} is due to the fact that they provide algebraic conditions to bound from below the fundamental unit which avoid the computation of the whole continued fraction of $\sqrt{d}$ or $\frac{1+\sqrt{d}}{2}$. For an effective version of \cite[Theorem 3.1]{Yamamoto} we refer to the work of Reiter \cite{Reiter}, while for an application we suggest \cite{JacWill}, where Jacobson and Williams compared the fundamental units of consecutive Pell equations. 

Our paper is organized as follows: in Section \ref{Section thm simultaneo} we will prove Theorem \ref{thm simultaneo} by estimating the $L$-function, the number of squarefree values and the fundamental unit;
Section \ref{prenot} gives a brief review of the basic theory of quadratic orders, quadratic irrationals and continued fractions; we end the article with Section \ref{proof HKY criterion}, devoted to the proof of Theorem \ref{HKY criterion}.


\section{Proof of Theorem \ref{thm simultaneo}}\label{Section thm simultaneo}
We now outline the general strategy. By \cite[Proposition 3.2]{Lamzouri}, if $d$ is a fundamental discriminant we are able to approximate $L(1,\chi_d)$ as a truncated Euler product. Our aim is to obtain 
\begin{equation*}
L(1,\chi_d)\ll (\log{\log{d}})^{-1},
\end{equation*}
i.e. the smallest possible asymptotic bound for the $L$-function if one assumes GRH (see \eqref{bound L}). We take $n>0$ from a suitable arithmetic progression and we put $N=n^2$ in $d=N+4p$, with $p$ a prime. By our particular choice of $n$, the character $\chi_d$ will assume the value $-1$ over many small primes. This will allow us to minimize the $L$-function. 
Then, through the small sieve, we count the number of squarefree values of $d$ occurring in the subfamily determined by the arithmetic progression. At this point we provide an estimate from below for $\xi_d$ through a calculation performed by following Yamamoto's proof. Indeed, if $d=n^2+4p$ we know that $\log{\xi_d}\gg (\log{d})^2$ with an implicit constant depending on $p$. We make the dependence explicit and obtain the constants stated in Theorem \ref{thm simultaneo}.

\subsection{The \texorpdfstring{$L$}{L}-function}

A good approximation to $L(1,\chi)$, for $\chi$ a primitive character, was obtained by Granville and Soundararajan \cite[Proposition 2.2]{GranvilleSound} using zero density estimates together with the large sieve. Later Lamzouri \cite{Lamzouri} improved their result by restricting to quadratic characters. We recall the statement:

\begin{prop}{\cite[Proposition 3.2]{Lamzouri}}\label{prop L Lam}
Let $A>1$ fixed. Then for all but at most $x^{1/A+o(1)}$ fundamental discriminants $0<d<x$ we have
\begin{equation}\label{eq L Lam}
L(1,\chi_d)=\prod_{p\leq (\log{x})^A}\left(1-\frac{\chi_d(p)}{p}\right)^{-1} \left( 1+O\left( \frac{1}{\log{\log{x}}} \right)  \right).
\end{equation}
\end{prop}

Primes up to a high power of $\log{x}$ might be hard to handle. However, as pointed out in \cite{Cherubini}, if $\varepsilon_1\in(1/2,1)$ equation \eqref{eq L Lam} implies
\begin{equation*}
L(1,\chi_d)\asymp_{\varepsilon_1}\prod_{p\leq (\log{x})^{\varepsilon_1}}\left(1-\frac{\chi_d(p)}{p}\right)^{-1}.
\end{equation*}
This shows that bounding trivially the primes $(\log{x})^{\varepsilon_1}<p\leq (\log{x})^A$ does not affect the final result. Then, we can consider just the primes $p$ up to $(\log{x})^{\varepsilon_1}$. From now on $p_i$ for $i=1,\ldots, m$ are primes as in Theorem \ref{thm simultaneo}.

Define
\begin{equation*}
S=\{p\leq m^24^m+2\quad and\quad p\neq p_i\ \forall i \},\qquad \mathcal{P}=\{p_i\mid p_i\leq 2m\}
\end{equation*}
and
\begin{equation*}
S'=\{p\mid m^24^m+2<p\leq(\log{x})^{\varepsilon_1}\quad and\quad p\neq p_i\ \forall i\};
\end{equation*}
then choose as modulus
\begin{equation}\label{def q sim}
q=\prod_{p\in \mathcal{P}\cup S\cup S'}p.
\end{equation}

\begin{rmk}\label{lemma q}
We will use that $q=o(x^\alpha)$ for any $\alpha\in (0,1)$ in Subsection \ref{squarefree counting} in order to find many squarefree values for $d_i$ with $n$ in an arithmetic progression. Moreover, since we will need it later, we observe that $q\gg (\log{x})^a$ for any $a>0$.
\end{rmk}

\begin{lem}\label{lemma modulo sim}
Let $n$ be an integer parameter, let $d_i=n^2+4p_i$ for $i=1,\ldots,m$ (we write $N=n^2$) and let $q$ be as in \eqref{def q sim}. If condition $(*)$ of Theorem \ref{thm simultaneo} holds, then there exists an integer $0\leq n_0<q$ such that, if $n\equiv n_0 \pmod{q}$, 
\begin{itemize}
   \item[(a)] $\chi_{d_i}(p)=-1$ for every $i$ and for every prime $p\in S'$;
   \item[(b)] $d_i\not\equiv0\pmod{p}$ for every $i$ and for every $p|q$.
\end{itemize}
\end{lem}
Notice that, since $q$ is even, b) implies $d_i\equiv1 \pmod{4}$ for all $i$.

\begin{proof}[Proof of Lemma \ref{lemma modulo sim}]
Since over odd primes $\chi_d$ is precisely the Legendre symbol $\left(\frac{d}{p}\right)$, we first determine the odd primes $p$ for which there exists $y$ such that
\begin{equation}\label{conditions simultaneous families}
\left(\frac{y}{p}\right)=1\quad\mbox{and}\quad \left( \frac{y+4p_i}{p} \right)=-1
\end{equation}
for every $i$.
Let us define the characteristic function
\begin{equation*}
\phi_p(y):=\frac{1}{2^{m+1}}\left(\frac{y}{p}\right)\left(1+\left(\frac{y}{p}\right)\right) \prod_{i=1}^m \left[ \left(\frac{y+4p_i}{p}\right)\left(\left(\frac{y+4p_i}{p}\right)-1\right)\right]
\end{equation*}
which returns $1$ or $0$ according to whether $y$ satisfies \eqref{conditions simultaneous families} or not.

If $p\neq 2,p_1,\ldots, p_m$ then
\begin{equation*}
\sum_{y\in \mathbb{F}_p}\phi_p(y)=\frac{1}{2^{m+1}}\sideset{}{'}\sum \left(\left(\frac{y}{p}\right)+1\right) \prod_{i=1}^m \left(1-\left(\frac{y+4p_i}{p}\right)\right),
\end{equation*}
where $\sideset{}{'}\sum$ means we are summing over the classes $y\not\equiv 0,-4p_i \pmod{p}$ where $i=1,\ldots,m$; therefore, we are omitting at most $m+1$  residue classes. If we reintroduce the corresponding terms in the sum on the right, each of them contributes at most $2^m$.
Thus
\begin{align}
\sum_{y\in \mathbb{F}_p}\phi_p(y)&\geq \frac{1}{2^{m+1}} \sum_{y\in\mathbb{F}_p} \left(\left(\frac{y}{p}\right)+1\right) \prod_{i=1}^m \left(1-\left(\frac{y+4p_i}{p}\right)\right) - \frac{m+1}{2} \nonumber\\
&=\frac{1}{2^{m+1}}\sum_{y\in\mathbb{F}_p} \left(\left(\frac{y}{p}\right)+1\right) \sum_{j=0}^m\sum_{\substack{J\subset\{1,\ldots,m\}\\ |J|=j}}(-1)^j\prod_{t\in J}\left(\frac{y+4p_t}{p}\right)- \frac{m+1}{2}\nonumber\\
&=\frac{1}{2^{m+1}}\sum_{j=0}^m\sum_{\substack{J\subset\{1,\ldots,m\}\\ |J|=j}}(-1)^j\sum_{y\in\mathbb{F}_p}\left(\frac{\prod_{t\in J}(y+4p_t)}{p}\right) \label{somme generatrice 1}\\
&+\frac{1}{2^{m+1}}\sum_{j=0}^m\sum_{\substack{J\subset\{1,\ldots,m\}\\ |J|=j}}(-1)^j\sum_{y\in\mathbb{F}_p}\left(\frac{y\prod_{t\in J}(y+4p_t)}{p}\right)-\frac{m+1}{2} \label{somme generatrice 2}.
\end{align}
We need to estimate the sums \eqref{somme generatrice 1} and \eqref{somme generatrice 2}; when possible we will apply Weil's theorem, see \cite[Theorem 11.23]{Iwaniec}. 
If $\prod_{t\in J}(y+4p_t)$ is a square as a polynomial in $\mathbb{F}_p[y]$, then 
\begin{equation*}
\left(\frac{\prod_{t\in J}(y+4p_t)}{p}\right)\geq0
\end{equation*}
for every value of $y$, so that the sum over $y\in\mathbb{F}_p$ is at most $p$. Moreover, we remark that in this case $j$ must be even.
If $\prod_{t\in J}(y+4p_t)$ is not a square in $\mathbb{F}_p[y]$ then, by \cite[Theorem 11.23]{Iwaniec},
\begin{equation*}
\left|\sum_{y\in\mathbb{F}_p}\left(\frac{\prod_{t\in J}(y+4p_t)}{p}\right)\right|\leq (j-1)p^{1/2}.
\end{equation*}
Let us look at \eqref{somme generatrice 1}. Since $J=\emptyset$ always gives $p$, we get the bound
\begin{equation}\label{bin bound 1}
\sum_{j=0}^m\sum_{\substack{J\subset\{1,\ldots,m\}\\ |J|=j}}(-1)^j\sum_{y\in\mathbb{F}_p}\left(\frac{\prod_{t\in J}(y+4p_t)}{p}\right)\geq p- p^{1/2}\sum_{j=1}^m\binom{m}{j}(j-1)
\end{equation}
while, whereas by the binomial theorem we have
\begin{equation*}
\sum_{j=1}^m\binom{m}{j}(j-1)=m2^{m-1}-2^m+1,
\end{equation*}
the right-hand side in \eqref{bin bound 1} equals
\begin{equation}\label{bin bound 1.1}
p-p^{1/2}-m2^{m-1}p^{1/2}+2^mp^{1/2}.
\end{equation}
We now look at \eqref{somme generatrice 2}. We observe that, since $p\neq 2,p_i$, the product $y\prod_{t\in J}(y+4p_t)$ is never a square as a polynomial in $\mathbb{F}_p[y]$. Again by \cite[Theorem 11.23]{Iwaniec} and the binomial theorem, we get
\begin{equation}\label{bin bound 2}
\sum_{j=0}^m\sum_{\substack{J\subset\{1,\ldots,m\}\\ |J|=j}}(-1)^j\sum_{y\in\mathbb{F}_p}\left(\frac{y\prod_{t\in J}(y+4p_t)}{p}\right)\geq - m2^{m-1}p^{1/2}.
\end{equation}
Combining \eqref{bin bound 1}, \eqref{bin bound 1.1} and \eqref{bin bound 2}, we obtain
\begin{equation*}
\sum_{y\in\mathbb{F}_p}\phi_p(y)\geq \frac{p}{2^{m+1}}-\left(\frac{m-1}{2}+\frac{1}{2^{m+1}}\right)p^{1/2}-\frac{m+1}{2};
\end{equation*}
then if $p\geq m^24^m+3$ it is straightforward to see that the right-hand side is positive so that, \textit{a fortiori}, if $p\in S'$ then \eqref{conditions simultaneous families} has a solution.
 
Now if $p\in S$ we choose $n$ in any class such that every $d_i$ is not $0$ modulo $p$: for instance if $p>2$ we can take $n\equiv0\pmod{p}$ while if $p=2$ take $n\equiv 1\pmod{2}$. Thanks to $(*)$, there is a class $n_0$ modulo $p_j$ such that $n_0^2+4p_i\not\equiv 0 \pmod{p_j}$ for all $1\leq i\leq m$ and all $p_j\leq 2m$. An application of the Chinese reminder theorem ends the proof.
\end{proof}

As in the statement of Lemma \ref{lemma modulo sim}, throughout this section we will call $n_0$ and $q$ the parameters defining the arithmetic progression of $n$. Let us show that along this progression we can bound the value of the $L$-function.

\begin{lem}
Let $d_i=n^2+4p_i$ be squarefree with the parameter $n$ in the arithmetic progression given by $n_0$ and $q$ (see Lemma \ref{lemma modulo sim}) and assume \eqref{eq L Lam} holds for every $i=1,\ldots,m$. If $x\gg 1$, then
\begin{equation}\label{stima L simultanea}
L(1,\chi_{d_i})<C'(m)\frac{2e^{-M}}{\log\log{x}} \cdot \frac{A}{(\varepsilon_1)^2}(1+o(1) )
\end{equation}
where
\begin{equation*}
C'(m)=\prod_{p_i> m^24^m+2}\left( \frac{p_i+1}{p_i-1} \right)\prod_{p\leq m^24^m+2}\left( \frac{p+1}{p-1} \right)
\end{equation*}
and $M=0.26149\ldots$ is Merten's constant.
\end{lem}

\begin{proof}
Call $\mathcal{P}_0$ the set of primes $p_i\in (m^24^m+2, (\log{x})^{\varepsilon_1}]$. We factor $L(1,\chi_{d_i})$ as in \eqref{eq L Lam} and estimate separately the different terms of the product. First notice that
\begin{equation}\label{prod 1}
\prod_{p\leq m^24^m+2}\left( 1-\frac{\chi_{d_i}(p)}{p} \right)^{-1}\leq \prod_{p\leq m^24^m+2}\left( 1-\frac{1}{p} \right)^{-1},
\end{equation}
while by Lemma \ref{lemma modulo sim} (a)
\begin{equation}\label{prod 2}
\prod_{\substack{m^24^m+2<p\\ p\leq (\log{x})^{\varepsilon_1}}}\left( 1-\frac{\chi_{d_i}(p)}{p} \right)^{-1}\leq \prod_{p \in S'}\left( 1+\frac{1}{p} \right)^{-1} \prod_{p\in \mathcal{P}_0}\left( 1-\frac{1}{p} \right)^{-1}.
\end{equation}
Inequalities \eqref{prod 1} and \eqref{prod 2} imply
\begin{equation}\label{prod 3}
\prod_{p\leq (\log{x})^{\varepsilon_1}}\left( 1-\frac{\chi_{d_i}(p)}{p} \right)^{-1}\leq \prod_{\substack{p\leq m^24^m+2\\ or\ p\in\mathcal{P}_0}}\left( \frac{p+1}{p-1} \right)\prod_{p\leq (\log{x})^{\varepsilon_1}}\left( 1+\frac{1}{p} \right)^{-1}
\end{equation}
and, as $0<\mid t \mid\leq 1/2$ guarantees $\log{(1+t)}> t-t^2$, we get
\begin{equation*}
\prod_{p\leq (\log{x})^{\varepsilon_1}}\left( 1+\frac{1}{p} \right)^{-1}< 1.9\exp\left(- \sum_{p\leq (\log{x})^{\varepsilon_1}}\frac{1}{p} \right)\leq\frac{2e^{-M}}{\varepsilon_1\log\log{x}},
\end{equation*}
where we bounded $\exp\left(\sum \frac{1}{p^2}\right)\leq1.9$ and applied Merten's theorem for $x\gg1$.
Inserting the above in \eqref{prod 3} we deduce
\begin{equation}\label{prod 5}
\prod_{p\leq(\log{x})^{\varepsilon_1}}\left( 1-\frac{\chi_{d_i}(p)}{p} \right)^{-1}< \prod_{\substack{p\leq m^24^m+2\\ or\ p\in\mathcal{P}_0}}\left( \frac{p+1}{p-1} \right)\frac{2e^{-M}}{\varepsilon_1\log\log{x}}.
\end{equation}

Similarly
\begin{equation}\label{prod 6}
\prod_{\substack{(\log{x})^{\varepsilon_1}<p\\ p\leq (\log{x})^A}}\left( 1-\frac{\chi_{d_i}(p)}{p} \right)^{-1}\leq \prod_{\substack{(\log{x})^{\varepsilon_1}<p\\ p\leq (\log{x})^A}}\left( 1-\frac{1}{p} \right)^{-1}< \frac{A}{\varepsilon_1}\exp(o(1))
\end{equation}
and then for $x\gg1$, combining \eqref{eq L Lam}, \eqref{prod 5} and \eqref{prod 6}, we get
\begin{equation*}
L(1,\chi_{d_i})< \frac{2e^{-M}}{\log\log{x}}  \frac{A}{(\varepsilon_1)^2}(1+o(1))\prod_{p_i> m^24^m+2}\left( \frac{p_i+1}{p_i-1} \right)\prod_{p\leq m^24^m+2}\left( \frac{p+1}{p-1} \right).
\end{equation*}
\end{proof}

\subsection{Counting the squarefree}\label{squarefree counting}

We now want to prove that there are many discriminants in arithmetic progression that can be taken simultaneously squarefree.
If $d_i=(n_0+kq)^2+4p_i$, for every $i=1,\ldots, m$, we know there are at most two classes for $k$ modulo $p^2$ for which $d_i$ is a multiple of $p^2$. We will use the small sieve to cancel these classes and produce $m$ discriminants simultaneously squarefree.

\begin{prop}\label{simultaneous squarefree}
Fix $\varepsilon>0$ and let $n_0$ and $q$ be as in Lemma \ref{lemma modulo sim}. Then, if $d_i=(n_0+kq)^2+4p_i$ with $p_i$ as in Theorem \ref{thm simultaneo}, there exists a positive constant $c(\varepsilon, m, a)$ depending on $\varepsilon,m$ and $a$ such that, if $x> c(\varepsilon, m, a)$, the number of $k>0$ for which $d_i<x$ is squarefree for all $i$ is at least $x^{1/2-\varepsilon}$.
\end{prop}

\begin{proof}
By hypothesis $p_i\leq (\log{x})^a$ for all $i$. In order to have $d_i<x$, it is sufficient to take $0<k\leq K$, where $K:=\lfloor q^{-1}\sqrt{x-4(\log{x})^a}-1\rfloor\asymp x^{1/2}/q$. We also define $M:=n_0+Kq$.

If $p|d_i$, then $p>2m$ by Lemma \ref{lemma modulo sim}. As in the proof of \cite[Lemma 2.2]{Cherubini}, we will prove that there are at least $K^{1-\varepsilon}$ values of $k$ for which, for a suitable parameter $z>(\log{x})^{\varepsilon_1}$,
\begin{itemize}
   \item[(a)] $p$ does not divide any of the $d_i$ for all $2m<p< z$;
   \item[(b)] $p^2$ does not divide any of the $d_i$ for all $z\leq p\leq M/z^{1/2}$.
\end{itemize}

We will apply the small sieve. If one chooses $z=q^2(\log{x})^{4m}$, the number of $k\leq K$ such that none of the $d_i$ is multiple of $p$ for $p< z$, for $x\gg_m 1$ (and then $K$ big), is
\begin{equation}\label{crivello 1}
\gg_m K \prod_{2m<p< z}\left(1-\frac{2m}{p}\right)\gg_m \frac{K}{(\log{z})^{2m}};
\end{equation}
so far at least $K^{1+o(1)}$ integers satisfy (a).
We now want to count the values of $k$ for which none of the $d_i$ is multiple of $p^2$ for $z\leq p\leq\frac{M}{z^{1/2}}$. For any of these $p$ there are at most $2m$ classes to sieve modulo $p$. We can bound their number by $2m\left(1+\frac{K}{p^2}\right)$.
It follows that we have to cancel no more values than
\begin{equation}\label{crivello 2}
2m\sum_{z\leq p\leq M/z^{1/2}}\left(1+\frac{K}{p^2}\right)\ll 2m\left( \frac{M}{z^{1/2}\log{\left(\frac{M}{z^{1/2}}\right)}}+\frac{K}{z} \right).
\end{equation}
Therefore, for our choice of $z$ and for $x\gg_m 1$, the term in \eqref{crivello 1} grows faster than \eqref{crivello 2} and so the number of $k$ satisfying both (a) and (b) is at least $K^{1+o(1)}$.

Finally, suppose $p^2|d_i$ for some $i$; as remarked at the beginning of the proof $p|d_i\Rightarrow p>2m$. Since we can assume none of the $d_i$ is a square (see the following Remark \ref{rmk simultaneous squarefree}), we obtain $p^2r\leq d_i$ for a prime $r$. By (a) and (b) we know $p> \frac{M}{z^{1/2}}$ and $r\geq z$ and therefore $M^2<d_i<x$. As $x-M^2=O(q)$, if $x\gg_{\varepsilon,m,a} 1$ the remaining values for $k$ are at least 
\begin{equation*}
K^{1+o(1)}+O(q)\geq x^{1/2+o(1)}+O(q)\geq x^{1/2-\varepsilon}.\qedhere
\end{equation*}
\end{proof}

\begin{rmk}\label{rmk simultaneous squarefree}
In the end of the proof of Proposition \ref{simultaneous squarefree} we could assume none of the $d_i$ is a square. Indeed, consider the discriminants $d_i$ in the interval $(x^{1/2},x)$, for which it is sufficient to assume $k\geq \frac{x^{1/4}}{q}$; then the hypothesis $p_i\leq (\log{x})^a$ for all $i$ implies $n^2<n^2+4p_i<(n+1)^2$. 
This restriction does not affect the conclusion of Proposition \ref{simultaneous squarefree} because, if $x\gg_\varepsilon 1$, one gets $x^{1/2-\varepsilon/2}-\frac{x^{1/4}}{q}>x^{1/2-\varepsilon}$.
Moreover, the small sieve ensures that, if $p_i|d_i$, then $p_i\geq z$. Since $p_i\leq (\log{x})^a\ll q<z$ (see Remark \ref{lemma q}), $p_i$ is unramified in $\mathbb{Q}(\sqrt{d_i})$ (actually it splits) and we can make use of \cite[Theorem 3.1]{Yamamoto}.
\end{rmk}

We can now conclude the proof of Theorem \ref{thm simultaneo}
\begin{proof}[Proof of Theorem \ref{thm simultaneo}]
Call $p_{\max}$ the maximum of the $p_i$; this means $p_{\max}\leq (\log{x})^a$.
By Proposition \ref{simultaneous squarefree}, if $x\gg_{\varepsilon,m,a} 1$, we get $x^{1/2-\varepsilon}$ squarefree $m$-tuples with no repetitions in the entries. Since by Proposition \ref{prop L Lam} the decomposition of the $L$-function fails for at most 
$x^{1/A+o(1)}$ of the $m$-tuples, if $A>2$ it follows that there are at least
\begin{equation*}
x^{1/2-\varepsilon}-x^{1/A+o(1)}
\end{equation*}
$m$-tuples, with entries all smaller than $x$, such that one can decompose $L(1,\chi_{d_i})$ as a truncated Euler product as in \eqref{eq L Lam} for every $d_i$.

Concerning the size of the regulator, we can apply \cite[Theorem 3.1]{Yamamoto} (see Remark \ref{rmk simultaneous squarefree}) and deduce that, if $d=n^2+4p$, then
\begin{equation*}
\log{\xi_d}\geq \frac{(\log{d})^2}{4\log{p}} + R
\end{equation*}
where $|R| \leq 3\log{d}+2\log{p}+5\log{2}$. When $x$ is big, we can assume $d_i>x^{1/2}$ as shown in Remark \ref{rmk simultaneous squarefree} and, given that $p_i\leq (\log{x})^a$, one has
\begin{equation}\label{right xi estimate}
\log{\xi_{d_i}}\geq \frac{(\log{d_i})^2}{8\log{p_i}}
\end{equation}
for all $i$ and $x\gg_a 1$.

At this point by the class number formula \eqref{class number formula}, the bound on the $L$-function \eqref{stima L simultanea} and the bound on the regulator \eqref{right xi estimate}, we obtain
\begin{equation}\label{di inequality}
h(d_i)< C'(m) 2 e^{-M} \frac{A}{(\varepsilon_1)^2}(1+o(1)) 4\log{p_i}\frac{\sqrt{d_i}}{(\log{d_i})^2\log\log{d_i}}.
\end{equation}
Since $d_i>x^{1/2}$ and $p_{\max}\leq (\log{x})^a$, $N=n^2= d_i(1+o(1))$ uniformly in $i$ and this gives
\begin{equation}\label{dn inequality}
\frac{\sqrt{d_i}}{(\log{d_i})^2\log\log{d_i}}= (1+o(1))\frac{\sqrt{N}}{(\log{N})^2\log\log{N}}.
\end{equation}
By choosing $A>2$ sufficiently close to $2$, $\varepsilon_1$ close to $1$ and $x\gg_{\varepsilon,m,a} 1$, from inequalities \eqref{di inequality} and \eqref{dn inequality} we get
\begin{equation*}
h(d_i)\leq 16 C'(m)  \log{p_{\max}}\frac{\sqrt{N}}{(\log{N})^2\log\log{N}}.
\end{equation*}
The proof is complete.
\end{proof}


\begin{rmk}
The case of the single family is a particular case of Theorem \ref{thm simultaneo}, but the squarefree values admit a natural density. Namely, if $P(k)=(n_0+kq)^2+4p_0$, 
\begin{equation}\label{squarefree density}
\lim_{K\to\infty}\frac{\#\{0\leq k\leq K| P(k)\ \text{is squarefree}\}}{K}=\prod_p\left( 1-\frac{\rho(p^2)}{p^2}  \right)>0,
\end{equation}
where $\rho(p^2)$ is the number of solutions of $P(k)\equiv0$ modulo $p^2$.
Equation \eqref{squarefree density} also offers a probabilistic interpretation: each factor represents the probability of not being divisible by $p^2$ and thus, if one takes the product over each prime, the probability of being squarefree.
\end{rmk}


\section{Quadratic irrationals: a review}\label{prenot}
Here we collect some notation and statements about the classical theory of quadratic fields, quadratic irrationals and continued fractions, in order to facilitate the reading of the proof of Theorem \ref{HKY criterion}. We basically recall the content of \cite[§2]{H-K}, see also \cite{HKbook}.

\subsection{Quadratic orders}\label{prenot1}
A natural number $d$ is called a \textit{(quadratic) discriminant} if $d$ is not a square and satisfies $d \equiv 0 \pmod{4}$ or $ d \equiv 1 \pmod{4} $. 
In the following, let $d$ always be a discriminant. We define
\begin{equation*}
\omega_d = \begin{cases} 
\frac{\sqrt{d}}{2} & \text{if } d \equiv 0 \pmod{4}, \\
\frac{1 + \sqrt{d}}{2} & \text{if } d \equiv 1 \pmod{4}.
\end{cases}
\end{equation*}
Then, $\mathcal{O}_d = \mathbb{Z}[\omega_d] $ is an order in the real quadratic field $\mathbb{Q}(\sqrt{d}) $; if $d_0 $ is the fundamental discriminant of $\mathbb{Q}(\sqrt{d}) $, then $d= d_0 f^2 $ for $ f \in \mathbb{N} $. We call $ \mathcal{O}_d $ the \textit{quadratic order with discriminant $d$} and $ f $ the \textit{conductor} of $ \mathcal{O}_d $. Every $\alpha \in \mathcal{O}_d$ has a unique representation of the form
\begin{equation*}
\alpha = \frac{b + e \sqrt{d}}{2}\quad \mbox{with}\  b, e \in \mathbb{Z}\  \mbox{and} \  b \equiv ed \pmod{2}.
\end{equation*}

Every ideal $\mathfrak{a} \neq [0] $ of $ \mathcal{O}_d $ is a free $\mathbb{Z}$-module of rank 2, and hence has the form
\begin{equation}\label{modulo-ideal}
\begin{cases}
\mathfrak{a} = a \mathbb{Z} + \frac{b + e \sqrt{d}}{2} \mathbb{Z}\\
a,\ b,\ e \in \mathbb{Z},\  a > 0,\  e > 0 ,\ b \equiv ed \pmod{2}.
\end{cases}
\end{equation}
 Conversely, if a free $ \mathbb{Z}$-module $ \mathfrak{a} \subset \mathbb{Q}(\sqrt{d}) $ is given by \eqref{modulo-ideal} then $ \mathfrak{a} $ is an ideal of $ \mathcal{O}_d $ if and only if
\begin{equation}\label{cond modulo ideale}
e | a, \quad 2e | (ed - b), \quad \text{and} \quad 4ae | (b^2 - de^2).
\end{equation}
Then
\begin{equation*}
N(\mathfrak{a}) = |\mathcal{O}_d/\mathfrak{a}| = ae^2,
\end{equation*}
and $ e^{-1} \mathfrak{a} $ is also an ideal of $ \mathcal{O}_d $. If $ e = 1 $, then $\mathfrak{a} $ is called \textit{primitive}; in this case, $m^{-1} \mathfrak{a} $ for any $ m \in \mathbb{N} $ with $ m \geq 2 $ is not an ideal of $\mathcal{O}_d$.

\begin{lem}\label{lemma Omega}
Let $d$ be a quadratic discriminant and let
\begin{equation*}
\mathfrak{a}=\left[ a, \frac{b+\sqrt{d}}{2} \right]
\end{equation*}
be a primitive ideal of $\mathcal{O}_d$ with $(a,d)=1$. Then for all $t\in\mathbb{N}$ the ideal $\mathfrak{a}^t$ is primitive.
\end{lem}
\begin{proof}
We will proceed by induction on $t$ and, in order to compute the product of ideals, we partially recall the content of \cite[Theorem 5.4.6]{HKbook}: 
if 
\begin{equation*} 
\mathcal{I}_i=e_i\left[a_i, \frac{b_i+\sqrt{d}}{2} \right]
\end{equation*}
for $i=1,2,3$ are three ideals such that $\mathcal{I}_3=\mathcal{I}_1\mathcal{I}_2$, then $\mathcal{I}_3$ is primitive if and only if $e_1=e_2=e=1$, where $e:=\left(a_1,a_2, \frac{b_1+b_2}{2}\right)$. Moreover, in this case, 
\begin{equation*}
a_3=a_1a_2\quad \text{and}\quad b_3\equiv b_2\pmod{2a_2}.
\end{equation*}
Throughout the proof we will write  
\begin{equation*}
\mathfrak{a}^t=e_t\left[ a_t, \frac{b_t+\sqrt{d}}{2} \right];
\end{equation*}
obviously $a_1=a$ and $b_1=b$.
If $t=1$ the thesis is trivial. Let $\mathfrak{a}^t$ be primitive: to show that $\mathfrak{a}^{t+1}$ is also primitive, since $\mathfrak{a}^{t+1}=\mathfrak{a}^t\mathfrak{a}$ with $\mathfrak{a}^t$ and $\mathfrak{a}$ primitive, we should have $e_{t+1}=\left(a_t,a,\frac{b_t+b}{2}\right)=1$; since $\mathfrak{a}^t=\mathfrak{a}^{t-1}\mathfrak{a}$, it is easy to see that
\begin{equation*}
a_t\equiv 0 \pmod{a},\quad b_t\equiv b \pmod{2a}.
\end{equation*}
Thus, since $(a,d)=1$ implies $(a,b)=1$ by \eqref{cond modulo ideale},
\begin{equation*}
e_{t+1}=\left(a_t,a,\frac{b_t+b}{2}\right)=(a,b)=1.\qedhere
\end{equation*}
\end{proof}

Now take $a$, $b$ $\in \mathbb{Z}$, $a > 0$, $b^2 \equiv d \pmod{4a}$ and let $\mathfrak{a} = a \mathbb{Z} + \frac{b + \sqrt{d}}{2} \mathbb{Z} $ be a primitive ideal of $\mathcal{O}_d $. The ideal $\mathfrak{a} $ is called \textit{regular} if $\gcd\left(a,b,\frac{d - b^2}{4a}\right) = 1$.
If $ \mathfrak{a} $ is regular, then $ \mathfrak{a} $ is invertible, and for the conjugate ideal $ \mathfrak{a}' = a \mathbb{Z} + \frac{b- \sqrt{d}}{2} \mathbb{Z} $, we have
\begin{equation*}
\mathfrak{a} \cdot \mathfrak{a}' = a \mathcal{O}_d.
\end{equation*}
Two invertible ideals $ \mathfrak{a}_1 $ and $ \mathfrak{a}_2 $ are called equivalent if there exist $ \alpha_1,\alpha_2 \in \mathbb{Q}(\sqrt{d})^*$ such that $ \alpha_1 \mathfrak{a}_1 = \alpha_2 \mathfrak{a}_2 $. The equivalence classes of invertible ideals form an abelian group under multiplication, isomorphic to the class group $\mathrm{Pic}(\mathcal{O}_d)$, of order \( h(d) \). As the notation suggests, when $d=d_0$ is a fundamental discriminant, and thus $\mathcal{O}_d$ is the maximal order, every ideal is invertible and $h(d)$ is the usual class number.

A primitive ideal $ \mathfrak{a} $ of $ \mathcal{O}_d $ is called \textit{prime to the conductor $f$ } of $ \mathcal{O}_d$ if $\left(N(\mathfrak{a}), f\right)=1 $. Every primitive ideal prime to $f$ is regular, and so invertible, and each ideal class in $\mathrm{Pic}(\mathcal{O}_d)$ contains a primitive ideal prime to $f$.

The group of units of $\mathcal{O}_d $ is of the form
\begin{equation*}
\mathcal{O}_d^* = \{\pm \xi_d^k : k \in \mathbb{Z}\}
\end{equation*}
with a uniquely determined $\xi_d > 1 $, called the fundamental unit of $ \mathcal{O}_d $.

\subsection{Quadratic irrationals }\label{sub theta}
A real number $\rho$ is called a \textit{quadratic irrational} with discriminant $ d$  if
\begin{equation*}
\rho = \frac{b + \sqrt{d}}{2a}
\end{equation*}
for integers $ a, b $ with $ a > 0 $, $ 4a| b^2 - d  $ and $ \gcd\left(a, b, \frac{ b^2-d}{4a}\right) = 1$; $a$ and $b$ are uniquely determined.
Two quadratic irrationals $\rho$ and $\eta $ (with the same discriminant $ d $) are called equivalent if
\begin{equation*}
\eta = \frac{u \rho + v}{w \rho + z}
\end{equation*}
for integers $u, v, w, z $ such that $ uz - vw = 1 $. 
A quadratic irrational $\rho$ is called \textit{reduced} if
\begin{equation*}
\rho > 1, \quad -1 < \rho' < 0,
\end{equation*}
where $\rho'$ denotes the Galois conjugate of $\rho$. 

If we set
\begin{equation*}
\theta(\rho):=a \mathbb{Z} + \frac{b + \sqrt{d}}{2} \mathbb{Z},
\end{equation*}
then $\theta(\rho)$ is a regular ideal of $ \mathcal{O}_d $ with $N(\theta(\rho)) = a $, and every regular ideal of $ \mathcal{O}_d $ is of this form.
A regular ideal $\mathfrak{a}$ of $\mathcal{O}_d $ is called \textit{reduced} if there exists a reduced quadratic irrational $\rho$ such that $\mathfrak{a} = \theta(\rho)$.
If $\mathfrak{a}$ is regular and $N(\mathfrak{a})<\frac{1}{2}\sqrt{d}$ then it is reduced and for every reduced ideal $\mathfrak{a}$ we have $N(\mathfrak{a})<\sqrt{d}$. 

\begin{rmk}\label{rmk prim red}
By what we have seen up to now, if $\mathfrak{a}$ is primitive with $(N(\mathfrak{a}),f)=1$ and $N(\mathfrak{a})<\frac{1}{2}\sqrt{d}$ then $\mathfrak{a}$ is regular with $N(\mathfrak{a})<\frac{1}{2}\sqrt{d}$ which implies that $\mathfrak{a}$ is reduced. Since by definition a reduced ideal is primitive, if $\mathfrak{a}$ is prime to the conductor $f$ and its norm is $<\sqrt{d}/2$ then $\mathfrak{a}$ is reduced if and only if $\mathfrak{a}$ is primitive.
\end{rmk}

Two quadratic irrationals $\rho$ and $\eta $ are equivalent if and only if $\theta(\rho)$ and $\theta(\eta)$ are equivalent ideals of $ \mathcal{O}_d$. Thus, $\theta$ induces a bijection from the set of equivalence classes of quadratic irrationals with discriminant $d$ to $\mathrm{Pic}(\mathcal{O}_d)$. Moreover with our definitions $\theta$ descends to a bijection from the set of reduced quadratic irrationals with discriminant $ d$ to the set of reduced ideals of $ \mathcal{O}_d $.
In each class of $\mathrm{Pic}(\mathcal{O}_d)$ there is a reduced ideal and thus $\theta$ induces a bijection between the set of equivalence classes of reduced quadratic irrationals of discriminant $d$ and $\mathrm{Pic}(\mathcal{O}_d)$.

\subsection{Continued fractions}

A real number $\rho$ is a reduced quadratic irrational if and only if $\rho$ has a purely periodic continued fraction expansion. If $\rho$ is a quadratic irrational with discriminant $d$ and its continued fraction expansion is
\begin{equation*}
\rho = [b_0, b_1, \dots, b_{k-1}, \overline{b_k, \dots, b_{k+l-1}}],
\end{equation*}
with minimal period $ l $, then the $ l $ reduced quadratic irrationals
\begin{equation*}
\rho_\nu = [\overline{b_{k+\nu-1}, \dots, b_{k+l-1}, b_k,\ldots, b_{k+\nu-2}}] \quad (\nu= 1, 2, \dots, l )
\end{equation*}
are exactly the reduced quadratic irrationals equivalent to $\rho$.

Let $\rho = \omega_d$; then $\rho = [b_0, \overline{b_1, \ldots, b_l}]$
with period $ l $. Since $\theta(\omega_d)= [1]\in \mathrm{Pic}(\mathcal{O}_d) $, the ideals $ \theta(\rho_1), \ldots, \theta(\rho_l) $ are precisely the reduced principal ideals of $ \mathcal{O}_d $. The fundamental unit $\xi_d$ of $\mathcal{O}_d $ is obtained by
\begin{equation*}
\xi_d = \prod_{\nu=1}^l \rho_\nu,
\end{equation*}
and we have $N(\xi_d)=(-1)^l$, see \cite[Proposition 1.2]{Yamamoto} or \cite[Theorem 2.2.9]{HKbook}.

\section{Halter-Koch--Yamamoto criterion}\label{proof HKY criterion}

We follow the proof of Halter-Koch; we correct the critical step in \cite[§3]{H-K} providing some further considerations.

\begin{proof}[Proof of Theorem \ref{HKY criterion}]
We can first assume, without loss of generality, that the decompositions $n_i = n_i' n_i''$ are identical for all quadratic discriminants $d\in\mathfrak{D}$ (since we have a finite number of such possible decompositions, we can select one which holds for infinitely many $d$).

Let now $d\in\mathfrak{D}$ and let $\mathfrak{n}_i$ be the reduced principal ideals of $\mathcal{O}_d$ with $N(\mathfrak{n}_i) = n_i$. Then $ \mathfrak{n}_i = \mathfrak{n}_i ' \mathfrak{n}_i '' $ with primitive ideals $\mathfrak{n}_i ' $ and $\mathfrak{n}_i '' $ in $\mathcal{O}_d$ such that $N(\mathfrak{n}_i ') = n_i' $ and $N(\mathfrak{n}_i '') = n_i''$. Indeed, as $(n_i',n_i'')=1$, by a straightforward computation one has
\begin{equation*}
\mathfrak{n}_i=\left[ n_i'n_i'', \frac{b_i+\sqrt{d}}{2} \right]=\left[ n_i', \frac{b_i+\sqrt{d}}{2} \right]\left[ n_i'', \frac{b_i+\sqrt{d}}{2} \right]
\end{equation*}
for an integer $b_i$.

Due to condition (b2), we claim that $\mathfrak{n}_i ''^2 = n_i'' \mathcal{O}_d $, so that $\mathfrak{n}_i ''^2$ is a principal ideal of $\mathcal{O}_d $ for all $i$. Indeed, by the conditions \eqref{cond modulo ideale}, we have $2| b_i^2+d$ and $n_i''| b_i^2-d$. Since $n_i''$ divides $d$ by (b2), we also get $n_i''| b_i$. Thus, if $d\equiv 0 \pmod{4}$,
\begin{equation*}
\frac{b_i^2+d}{2n_i''}\in\mathbb{Z} \Longrightarrow \left(\frac{b_i+\sqrt{d}}{2}\right)^2\in n_i''\mathcal{O}_d
\end{equation*}
while, if $d\equiv 1 \pmod{4}$,
\begin{equation*}
\frac{b_i^2-2b_i+d}{4n_i''}\in\mathbb{Z} \Longrightarrow \left(\frac{b_i+\sqrt{d}}{2}\right)^2\in n_i''\mathcal{O}_d.
\end{equation*}
Therefore
\begin{equation*}
\mathfrak{n}_i ''^2=\left[ n_i''^2, \left(\frac{b_i+\sqrt{d}}{2}\right)^2, n_i''\frac{b_i+\sqrt{d}}{2}\right]\subset n_i'' \mathcal{O}_d
\end{equation*}
and the other inclusion is trivial (just look at the norms).
We can now deduce that $\mathfrak{n}_i '^2$ is a principal ideal with norm $n_i'^2$ and also primitive since $(b_i,n_i')=1$. By (b1) $(n_i'^2,f)=1$ and so $\mathfrak{n}_i '^2$ is regular; as $n_i'^2<\frac{1}{2}\sqrt{d}$, which we can assume because $d\to\infty$, $\mathfrak{n}_i '^2$ is also reduced.

We can then replace $n_i$ by $n_i'^2$ in the hypotheses of Theorem \ref{HKY criterion} and from now on we will assume $n_i''=1$.

Let all $n_i'' = 1$, so that $n_i= n_i' = N(\mathfrak{n}_i )$ with reduced principal ideals $\mathfrak{n}_i $ of $\mathcal{O}_d$. Then we define a set $\Omega$ of reduced principal ideals
\begin{equation*}
\Omega:=\left\{ \prod_{i=1}^m\mathfrak{n}_i ^{e_i}| e_i\geq0, \prod_{i=1}^m n_i^{e_i}<\frac{1}{2}\sqrt{d} \right\}.
\end{equation*}
These ideals are clearly principal; we will prove they are reduced too.
Indeed, if $I\in\Omega$, then by (b1) we have $(N(I),f)=1$, so it is equivalent to show $I$ is primitive, see Remark \ref{rmk prim red}. Since the product of primitive ideals with coprime norms is primitive and for $i\neq j$ the norms of $\mathfrak{n}_i ^{e_i}$ and $\mathfrak{n}_j^{e_j}$ are coprime by (b3), we just need to justify that $\mathfrak{n}_i ^{e_i}$ is primitive for every $i$; this follows by a straightforward application of Lemma \ref{lemma Omega}.

We now proceed to estimate the fundamental unit.
Take $\mathfrak{a}\in \Omega$ and let $\rho_\mathfrak{a}$ be the reduced quadratic irrational such that $\theta(\rho_\mathfrak{a})=\mathfrak{a}$, where $\theta$ is the bijection introduced in Subsection \ref{sub theta}. Call $\mathcal{H}$ the set of reduced quadratic irrationals $\rho$ for which $\theta(\rho)$ is principal in $\mathcal{O}_d$. Then
\begin{equation*}
\xi_d=\prod_{\rho\in\mathcal{H}}\rho \geq \prod_{\mathfrak{a}\in\Omega}\rho_\mathfrak{a};
\end{equation*}
for $\mathfrak{a}=\prod_{i=1}^m\mathfrak{n}_i ^{e_i}\in\Omega$ let $\rho_\mathfrak{a}=\frac{b+\sqrt{d}}{2A}$ with $A=\prod_{i=1}^m n_i^{e_i}$ and $b\in\mathbb{Z}$. 
Because $-1<\rho_\mathfrak{a}'<0$, one has that $b>0$; thus
\begin{equation*}
\rho_\mathfrak{a}> \frac{\sqrt{d}}{2}\prod_{i=1}^m n_i^{-e_i}.
\end{equation*}
This leads to
\begin{equation}\label{discrete estimate}
\log{\xi_d}\geq \sum_{\mathfrak{a}\in\Omega}\log{\rho_\mathfrak{a}}>\sum_{(e_1,\ldots,e_m)\in\mathbb{Z}^m\cap\mathcal{G}}\left( \log{\frac{\sqrt{d}}{2}}-\sum_{i=1}^m e_i\log{n_i} \right)
\end{equation}
with $\mathcal{G}:=\mathcal{G}\left(\log{\frac{\sqrt{d}}{2}}\right)$ where
\begin{equation}\label{def G}
\mathcal{G}(y):=\left\{ (x_1,\ldots, x_m)\in \mathbb{R}^m \mid x_i\geq 0, \sum_{i=1}^m x_i\log{n_i}\leq y \right\}.
\end{equation}
As we will prove in Lemma \ref{asymptotic lemma}, the right-hand term of \eqref{discrete estimate} can be approximated by an integral; more precisely, it can be written as
\begin{equation}\label{eq integrale I}
\int_\mathcal{G}\left( \log{\frac{\sqrt{d}}{2}}-\sum_{i=1}^m (x_i\log{n_i})  \right)dx_1\ldots dx_m +O\left((\log{d})^m\right).
\end{equation}
If we let $I$ denote the integral above, we see that
\begin{equation*}
I=\frac{1}{m!P}\left(\log{\frac{\sqrt{d}}{2}}\right)^{m+1}-\frac{1}{P} \sum_{i=1}^m \int_\mathcal{G'}x_i d\underline{x},
\end{equation*}
where $d\underline{x}=dx_1\ldots dx_m$, $P=\prod_i \log{n_i}$ and
\begin{equation*}
\mathcal{G'}=\left\{ (x_1,\ldots, x_m)\in \mathbb{R}^m \mid x_i\geq 0, \sum_{i=1}^m x_i\leq \log{\frac{\sqrt{d}}{2}} \right\};
\end{equation*}
hence
\begin{equation}\label{eq integrale I 2}
I=\frac{1}{P}\left(\log{\frac{\sqrt{d}}{2}}\right)^{m+1}\left(\frac{1}{m!}-\frac{1}{(m+1)!}\right)=\frac{m}{(m+1)!P}\left(\log{\frac{\sqrt{d}}{2}}\right)^{m+1}.
\end{equation}
In particular this shows that the second term in \eqref{eq integrale I} is a true error term for $d\to\infty$.
Combining \eqref{eq integrale I 2} with \eqref{discrete estimate} we get
\begin{equation*}
\log{\xi_d}\gg (\log{\sqrt{d}})^{m+1}.\qedhere
\end{equation*}
\end{proof}

We are left to prove Lemma \ref{asymptotic lemma}.
\begin{lem}\label{asymptotic lemma}
Let $m\geq1$, let $\mathcal{G}=\mathcal{G}\left(\log{\frac{\sqrt{d}}{2}}\right)$ be as in \eqref{def G} and let $n_i\geq 2$ be integers for $i=1,\ldots, m$. Then
\begin{equation*}
\sum_{(e_1,\ldots,e_m)\in\mathbb{Z}^m\cap\mathcal{G}}\left( \log{\frac{\sqrt{d}}{2}}-\sum_{i=1}^m e_i\log{n_i} \right)= I +O\left((\log{d})^m\right)
\end{equation*}
where
\begin{equation*}
I=\int_\mathcal{G}\left( \log{\frac{\sqrt{d}}{2}}-\sum_{i=1}^m (x_i\log{n_i})  \right)dx_1\ldots dx_m.
\end{equation*}
\end{lem}
\begin{proof}
Let $e=(e_1,\ldots,e_m)$ and $S=\sum_{i=1}^m\log{n_i}$, then one has
\begin{align}\label{lemma 4.1 primi conti}
&I - \sum_{e\in\mathbb{Z}^m\cap\mathcal{G}}\left( \log{\frac{\sqrt{d}}{2}}-\sum_{i=1}^m e_i\log{n_i} \right) \\
={}&(\mathrm{vol}(\mathcal{G})-|\mathbb{Z}^m\cap\mathcal{G}|) \log{\frac{\sqrt{d}}{2}}- \left(  \int_\mathcal{G}\sum_{i=1}^m(x_i\log{n_i})d\underline{x}- \sum_{e\in\mathbb{Z}^m\cap\mathcal{G}}\sum_{i=1}^m e_i\log{n_i}   \right)\nonumber,
\end{align}
where the difference $(\mathrm{vol}(\mathcal{G})-|\mathbb{Z}^m\cap\mathcal{G}|)$ can be bounded by $O((\log{d})^{m-1})$.
In the remaining terms, we approximate the integral from above and below by writing
\begin{align*}
&\sum_{e\in\mathbb{Z}^m\cap\mathcal{G}(\log{\sqrt{d}/2-S})}\sum_{i=1}^m e_i \log{n_i}\leq \int_\mathcal{G}\sum_{i=1}^m(x_i\log{n_i})d\underline{x} \\
\leq &\sum_{e\in\mathbb{Z}^m\cap\mathcal{G}}\sum_{i=1}^m e_i \log{n_i} + S\cdot \mathrm{vol}(\mathcal{G})=\sum_{e\in\mathbb{Z}^m\cap\mathcal{G}}\sum_{i=1}^m e_i \log{n_i} + O\left((\log{d})^m\right).
\end{align*}
The sum on the right is equal to the last term of \eqref{lemma 4.1 primi conti}, while the sum on the left differs from that in \eqref{lemma 4.1 primi conti} by the fact that we sum on a smaller set. Hence, we have to bound
\begin{equation*}
\sideset{}{'}\sum_{e\in\mathbb{Z}^m}\sum_{i=1}^m e_i \log{n_i},
\end{equation*}
where $\sum'$ means we are summing on the vectors $e\in\mathbb{Z}^m$ such that 
\begin{equation*}
\log{\frac{\sqrt{d}}{2}}- S<\sum_{i=1}^m e_i \log{n_i}\leq \log{\frac{\sqrt{d}}{2}}.
\end{equation*}
The number of such vectors is bounded from above by 
\begin{equation*}
\mathrm{vol}(\mathcal{G}(\log{\sqrt{d}/2+S}))-\mathrm{vol}(\mathcal{G}(\log{\sqrt{d}/2-S}))=O\left((\log{d})^{m-1}\right)
\end{equation*} 
and therefore
\begin{equation*}
\sideset{}{'}\sum_{e\in\mathbb{Z}^m}\sum_{i=1}^m e_i \log{n_i}=O\left((\log{d})^m\right).
\end{equation*}
The proof of the lemma is complete.
\end{proof}

\begin{rmk}
In \cite{H-K} Halter-Koch announced to have found a family of discriminants $d$ with regulator $\gg(\log{d})^4$, i.e. 
\begin{equation}\label{d log4}
d=(lp^kq+c)^2+4p^kq
\end{equation}
with $(q,c)=1$, $l=1$, $p=rs$ and $q=tp+r$ such that $t\geq1$, $r\geq2$, $s\geq2$ and $rs$, $s(ts+1)$, $r(ts+1)$ are multiplicatively independent.
The idea was to apply his formulation of Theorem \ref{HKY criterion} with $m=3$, but the critical step in his proof consists in showing that the principal ideals in the set $\Omega$ are reduced: the condition that the norms are less than $\sqrt{d}/2$ is sufficient if they are primitive, but the product of two primitive ideals is not always primitive (see \cite[Theorem 5.4.6]{HKbook}). Indeed, as we will explain below, the hypotheses he chose in \cite[§3]{H-K} are not enough to guarantee this. Moreover, we could not say if \eqref{d log4} satisfies the hypotheses of Theorem \ref{HKY criterion} for $m=3$.

In his proof Halter-Koch considered the ideals 
\begin{align*}
&\mathfrak{n}_1=\left[ rs, \frac{p^kq-c+\sqrt{d}}{2} \right], \quad \mathfrak{n}_2=\left[ r(ts+1), \frac{p^kq+c+\sqrt{d}}{2} \right],\\
&\mathfrak{n}_3=\left[ s(ts+1), \frac{p^kq+p+1-2s(ts-t+1)+\sqrt{d}}{2} \right].
\end{align*}
The computations in the proof of \cite[Proposition 1]{H-K} show, for instance, that $\mathfrak{n}_1$ and $\mathfrak{n}_2$ are principal reduced and 
\begin{equation*}
\mathfrak{n}_1\mathfrak{n}_2=r \left[ A, \frac{B+\sqrt{d}}{2} \right]
\end{equation*}
for suitable integers $A,B$. Despite by \cite[Theorem 5.4.6]{HKbook} $N(\mathfrak{n}_1\mathfrak{n}_2)<\sqrt{d}/2$, the ideal $\mathfrak{n}_1\mathfrak{n}_2$ is not primitive. This observation tells us that in the definition of $\Omega$ the products containing $\mathfrak{n}_1\mathfrak{n}_2$ must be excluded. Even if we include all other products, i.e. 
\begin{equation*}
\Omega:=\left\{ \prod_{i=1}^3\mathfrak{n}_i^{e_i}| e_i\geq0,\ e_1 e_2=0,\  \prod_{i=1}^m n_i^{e_i}<\frac{1}{2}\sqrt{d} \right\},
\end{equation*}
the lower bound on the regulator decreases to $\gg(\log{d})^3$ instead of $\gg(\log{d})^4$, and it drops to $\gg(\log{d})^2$ if there are not mixed products but only powers of single ideals. In order to justify this, we decompose $\Omega$ as a union of subsets according to the different kind of products one gets.
Then, respectively for $j=1,2,3$ and $(j,k)=(1,3),(2,3)$, we define
\begin{align*}
\Omega_j&:=\left\{ \mathfrak{n}_j^{e_j}\mid e_j\geq0,  n_j^{e_j}<\frac{1}{2}\sqrt{d} \right\}\\
\Omega_{j,k}&:=\left\{ \mathfrak{n}_j^{e_j}\mathfrak{n}_k^{e_k}\mid e_j,e_k\geq0, n_j^{e_j} n_k^{e_k}<\frac{1}{2}\sqrt{d} \right\}.
\end{align*}
Our assertion is proved if we can show that $\Omega_j$ and $\Omega_{j,k}$ give a contribution to the growth of $\xi_d$ respectively of order $(\log{d})^2$ and $(\log{d})^3$, which follows from the computations of the last part of the proof of Theorem \ref{HKY criterion}.
In other words, the products of $m$ ideals, $m=3$ in this case, fill the volume of the $m$-simplex in $\mathbb{R}^m$, while products of fewer ideals lie on the faces; thus they give a smaller contribution.
Therefore, in \cite[Proposition 1 (ii)]{H-K} we can have at best $\gg (\log{d})^3$ and we can in fact drop to $\gg (\log{d})^2$ if we choose the parameters so that $(r,ts+1)\neq1$ and $(r,s)\neq1$.
\end{rmk}

\subsection{Acknowledgements}
The author would like to thank his supervisor Giacomo Cherubini for his guidance, as well as for many helpful discussions and comments during the preparation of this article.

\bibliographystyle{plain}
\bibliography{biblio}

\end{document}